\documentclass[a4paper]{amsart}
\usepackage{graphicx}
\usepackage{amssymb} 
\usepackage{stmaryrd}
\usepackage[mathcal]{euscript}
\usepackage{tikz}
\usepackage{tikz-cd}
\usepackage{hyperref}
\hypersetup{%
  bookmarksnumbered=true,%
  colorlinks=true,%
  linkcolor=blue,%
  citecolor=blue,%
  filecolor=blue,%
  menucolor=blue,%
  urlcolor=blue,%
  bookmarksopen=true,%
  bookmarksdepth=2,%
  pageanchor=true}

\makeatletter
\@namedef{subjclassname@2020}{\textup{2020} Mathematics Subject Classification}
\makeatother

\title{Chase's lemma and its context}

\author{Henning Krause}
\address{Fakult\"at f\"ur Mathematik\\
Universit\"at Bielefeld\\ D-33501 Bielefeld\\ Germany}
\email{hkrause@math.uni-bielefeld.de}

\theoremstyle{plain}
\newtheorem{thm}{Theorem}

\newtheorem{lem}[thm]{Lemma}

\theoremstyle{definition}

\theoremstyle{remark}
\newtheorem{rem}[thm]{Remark}

\numberwithin{equation}{section}

\newcommand{\Add}{\operatorname{Add}}

\newcommand{\card}{\operatorname{card}}

\newcommand{\fp}{\operatorname{fp}}
\newcommand{\Fp}{\operatorname{Fp}}

\newcommand{\Hom}{\operatorname{Hom}}

\newcommand{\Inj}{\operatorname{Inj}}

\newcommand{\Lex}{\operatorname{Lex}}

\renewcommand{\mod}{\operatorname{mod}}
\newcommand{\Mod}{\operatorname{Mod}}

\newcommand{\Prod}{\operatorname{Prod}}

\newcommand{\Ab}{\mathrm{Ab}}

\newcommand{\op}{\mathrm{op}}

\newcommand{\longiso}{\xrightarrow{\ \raisebox{-.4ex}[0ex][0ex]{$\scriptstyle{\sim}$}\ }}

\newcommand{\lto}{\longrightarrow}

\newcommand{\xto}{\xrightarrow}
\newcommand*{\intref}[2]{\def\tmp{#1}\ifx\tmp\empty\hyperref[#2]{\ref*{#2}}\else\hyperref[#2]{#1~\ref*{#2}}\fi}

\def\A{\mathcal A} 
 
\def\C{\mathcal C}
\def\D{\mathcal D}

\def\T{\mathcal T}

\def\bfP{\mathbf P}

\def\bbN{\mathbb N}

\def\a{\alpha}

\def\g{\gamma}
\def\p{\phi}

\def\k{\kappa}

\def\La{\Lambda}
\def\Si{\Sigma}

\begin{document}



\begin{abstract}
  Chase's lemma provides a powerful tool for translating properties of
  (co)products in abelian categories into chain conditions. This note
  discusses the context in which the lemma is used, making explicit
  what is often neglected in the literature because of its technical
  nature.
\end{abstract}
  
\date{May 3, 2021}
\maketitle

\section{The context of Chase's lemma}
In this note we discuss a technical lemma due to Chase
\cite{Ch1960,Ch1962} which provides
a relation between direct products and direct sums of modules. This
lemma has several interesting consequences. Chase used this for the
study of products of projective modules, but it was then noticed by
Gruson and Jensen \cite{GJ1976} that it also applies to the study of
$\Si$-pure-injective modules. Recall that a module $X$ is
\emph{$\Si$-pure-injective} if any coproduct of copies of $X$
is pure-injective.

\begin{thm}[Gruson--Jensen]\label{th:modules}
  For a module $X$ over a ring the following conditions are equivalent:
\begin{enumerate}
  \item The module $X$ is $\Si$-pure-injective.
  \item There exists a cardinal $\k$ such that every product of copies
    of $X$ is a pure submodule of a coproduct of modules of
    cardinality at most $\k$.
  \item There exists a module $Y$ such that every product of copies of
    $X$ is a pure submodule of a coproduct of copies of $Y$.
    \end{enumerate}
  \end{thm}

  The equivalence (1) $\Leftrightarrow$ (2) is stated in
  \cite{GJ1976} and the proof is rather short; it says (2)
  $\Rightarrow$ (1) \emph{s'obtiennent par extension des m\'ethodes de
    \cite{Ch1960}}. Further equivalent conditions are formulated in
  \cite{GJ1976} and also studied in independent work by Zimmermann
  \cite{Zi1977} and Zimmermann-Huisgen \cite{ZH1979}. There are many
  references to this result in the literature, but it took more than
  20 years until a full proof was published by Huisgen-Zimmermann
  \cite[Theorem~10]{HZ2000} in a special volume devoted to infinite
  length modules \cite{KR2000}, using Chase's lemma.

  Condition (3) is actually useful in other categorical settings
  where no obvious notion of cardinality is available. Clearly, (2)
  and (3) are equivalent, because the isomorphism classes of modules
  of cardinality bounded by $\k$ form a set and we can take the
  coproduct of a set of representatives.
    
  Replacing elements of modules with morphisms, Chase's lemma can be
  formulated more generally for abelian categories, cf.\
  Lemma~\ref{le:chase}. Then one obtains as a consequence a
  characterisation of locally noetherian Grothendieck categories which
  is due to Roos \cite{Ro1968}. In particular, we see that
  \emph{properties of (co)products translate into chain conditions};
  this seems to be the real essence of Chase's lemma.

\begin{thm}[Roos]\label{th:loc-noeth}
  A locally finitely generated Grothendieck category is locally
  noetherian (so has a generating set of noetherian objects) if and
  only if there is an object $E$ such that every object is a subobject
  of a coproduct of copies of $E$.
\end{thm}

For such a cogenerating object $E$ we have that every product of
copies of $E$ is a subobject of a coproduct of copies of $E$. Also, we
may assume that $E$ is injective, because one may replace $E$ with its
injective envelope. Then $E$ satisfies condition (3) in Theorem~\ref{th:modules}, and
this yields a first connection between the two theorems.

Getting back to work of Gruson and Jensen \cite{GJ1973} one knows that for any
ring $\La$ the fully faithful \emph{transfer functor}
\[T\colon \Mod\La\lto \Add(\mod(\La^\op),\Ab),\qquad X\mapsto
  X\otimes_\La-\] identifies pure-injective $\La$-modules with
injective objects in the category of additive functors
$\mod(\La^\op)\to\Ab$.  Here, $\Mod\La$ denotes the category of right
$\La$-modules and $\mod\La$ the full subcategory of finitely presented
modules. In particular, a $\La$-module $X$ is $\Si$-pure-injective if
and only if any coproduct of copies of $T(X)$ is injective.

In order to explain the relevance of Chase's lemma and the close
connection between the two theorems we adopt a more general approach,
following Crawley-Boevey \cite{CB1994}. We fix a locally finitely
presented additive category $\A$ and have a fully faithful functor
\[T\colon\A\lto\bfP(\A)\] into its \emph{purity category} $\bfP(\A)$
which is a locally finitely presented Grothendieck
category.\footnote{$\bfP(\A)=\Lex(\Fp(\fp\A,\Ab),\Ab)$, where $\fp\A$
  denotes the full subcategory of finitely presented objects in $\A$,
  $\Fp(\C,\Ab)$ denotes the abelian category of finitely presented
  functors $\C\to\Ab$, and $\Lex(\D,\Ab)$ denotes the category of
  additive functors $\D\to\Ab$ which send each short exact sequence in
  $\D$ to a left exact sequence in $\Ab$. The functor $T$ maps
  $X\in\fp\A$ to $\Hom(\Hom(X,-),-)$ and preserves filtered colimits.}
The functor $T$ preserves all (co)products and identifies
pure-injective objects in $\A$ with injective objects in $\bfP(\A)$
\cite[\S3]{CB1994}. For example, we can take $\A=\Mod\La$ for a ring
$\La$ and then the functor $\A\to\bfP(\A)$ identifies with the above
functor $X\mapsto X\otimes_\La-$.

Each object $X\in\A$ gives rise to a localising subcategory
$\C_X\subseteq \bfP(\A)$ that is generated by all finitely presented
objects $C\in\bfP(\A)$ satisfying $\Hom(C,T(X))=0$. We write $\Prod X$
for the full subcategory of products of copies of $X$ and their direct summands.

\begin{lem}
  An object $X\in\A$ is $\Si$-pure-injective if and only if the
  localised category $\bfP(\A)/\C_X$ is locally noetherian.  In this
  case $T$ induces an equivalence
\[\Prod X\longiso\Inj(\bfP(\A)/\C_X)\]
onto the full subcategory of injective objects in $\bfP(\A)/\C_X$.
\end{lem}

This lemma is useful because properties of $\Si$-pure-injective
objects (for example essentially unique decompositions into
indecomposable objects) can now be derived from a well developed
theory of injective objects in locally noetherian Grothendieck
categories \cite{Ga1962}.  For a proof we refer to
\cite[\S9]{Kr1998}, which combines the ideas from
\cite{CB1994,GJ1973,Si1978} with the localisation theory for
Grothendieck categories \cite{Ga1962}.

The above approach towards the study of pure-injectivity works equally
well for a compactly generated triangulated category $\T$ via the
\emph{restricted Yoneda functor}
\[\T\lto\Add((\T^c)^\op,\Ab),\qquad X\mapsto\Hom(-,X)\]
where $\T^c$ denotes the full subcategory of compact objects \cite{Kr2000}.

\section{Chase's lemma for additive categories}

In the context of modules over a ring, Chase's lemma goes back to an
argument in the proof of Theorem~3.1 in \cite{Ch1960}, though it is not
stated explicitly as a lemma.  In a subsequent paper \cite{Ch1962}
Chase formulated this as follows.

\begin{center}
  \includegraphics[width=.96\textwidth]{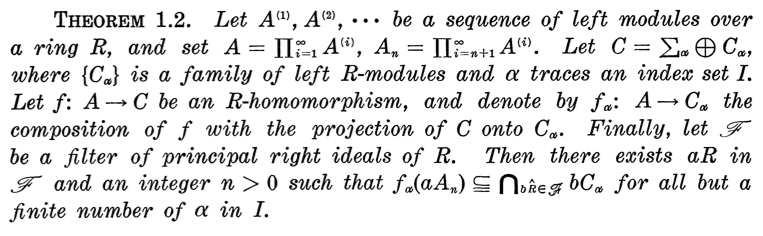}
\end{center}

We continue with a version of Chase's lemma for additive categories
which seems to be new.  For a sequence of morphisms
$\g=(C_{n}\rightarrow C_{n+1})_{n\in\bbN}$ we denote by
$\g_n\colon C_0\to C_n$ the composite of the first $n$ morphisms.  An
object $C$ is called \emph{finitely generated} if any morphism
$C\to \coprod_{i\in I}X_i$ factors through $\coprod_{i\in J}X_i$ for
some finite subset $J\subseteq I$.

\begin{lem}[Chase]\label{le:chase}
  Let $(X_n)_{n\in\bbN}$ and $(Y_i)_{i\in I}$ be families of objects
  in an additive category  and
  \[\p\colon \prod_{n\in\bbN}X_n\lto \coprod_{i\in I}Y_i\]
  a morphism. If $\g=(C_{n}\rightarrow C_{n+1})_{n\in\bbN}$ is
  a sequence of morphisms and $C=C_0$ is finitely generated, then
  there exists $m\in\bbN$ such that for almost all $j\in I$ each
  composite
\[C\xto{\ \g_m\ }C_m\xto{\ \theta\ }\prod_{n\in\bbN}X_n\xto{\ \p\ }\coprod_{i\in I}Y_i\twoheadrightarrow Y_j\]
with $\theta_n=0$ for $n<m$ factors through $\g_n\colon C\to C_n$ for all $n\in\bbN$.
\end{lem}

It is convenient to introduce further notation. For a morphism
$\g\colon C\to D$ and an object $X$ we denote by $X_\g$ the image of
the map
\[\Hom(D,X)\xto{\ -\circ\g\ }\Hom(C,X).\]
Then a sequence of morphisms
$\g=(C_{n}\rightarrow C_{n+1})_{n\in\bbN}$ yields a descending chain
\[\cdots\subseteq X_{\g_2}\subseteq X_{\g_1}\subseteq
  X_{\g_0}=\Hom(C_0,X).\]
We can now rephrase the statement of the lemma as follows.  \emph{There exists $m\in\bbN$ such that
  \[\p_i\big((\prod_{n\ge
      m}X_n)_{\g_m}\big)\subseteq\bigcap_{n\ge 0}(Y_i)_{\g_n}\] for
  almost all $i\in I$, where
  \[\p_i\colon\Hom(C,X)\xto{\ \p\circ -\ }\Hom(C,Y) \lto\Hom(C,Y_i).\]}

\begin{proof}
  We follow closely the proof of Theorem~1.2 in \cite{Ch1962}.  Assume
  the conclusion to be false. We set $X=\prod_{n\in\bbN}X_n$ and
  construct inductively sequences of elements $n_j\in\bbN$,
  $i_j\in I$, and $\theta_j\in\Hom(C,X)$ with $j\in\bbN$ and
  satisfying
  \begin{enumerate}
  \item  $n_{j+1}>n_j$,
    \item  $\theta_{j}\in (\prod_{n\ge n_{j}}X_n)_{\g_{n_j}}$,
\item $\p_{i_j}(\theta_j)\not\in (Y_{i_j})_{\g_{n_{j+1}}}$,
\item $\p_{i_j}(\theta_k)=0$ for $k<j$.
\end{enumerate}
We proceed as follows. Set $n_0=0$. Then there exists $i_0\in I$ such
that 
 \[\p_{i_0}(X_{\g_0})\not\subseteq\bigcap_{n\ge 0}(Y_{i_0})_{\g_n},\] 
 and hence we may select $\theta_0\in X_{\g_0}$ and $n_1>0$ such
 that $\p_{i_0}(\theta_0)\not\in (Y_{i_0})_{\g_{n_{1}}}$. Thus
conditions (1)--(4) are satisfied for $j=0$.

Proceeding by induction on $j$, assume that elements $n_{k+1}\in\bbN$, $i_k\in I$ and
$\theta_k\in\Hom(C,X)$ have been constructed for $k<j$ such that conditions
(1)--(4) are satisfied. Using that $C$ is finitely generated,  there exists a finite subset
$I'\subseteq I$ such that for $i\in I\setminus I'$ we have
$\p_i(\theta_k)=0$ for $k<j$. We may then select  $i_j\in I\setminus I'$
such that
  \[\p_{i_j}\big((\prod_{n\ge
      n_j}X_n)_{\g_{n_j}}\big)\not\subseteq\bigcap_{n\ge
      0}(Y_{i_j})_{\g_n},\] because otherwise the lemma would be true.
  Thus there exists $\theta_j\in (\prod_{n\ge n_{j}}X_n)_{\g_{n_j}}$
  and $n_{j+1}>n_j$ such that
  $\p_{i_j}(\theta_j)\not\in (Y_{i_j})_{\g_{n_{j+1}}}$. It is then
  clear that the elements $n_{k+1}\in\bbN$, $i_k\in I$, and
  $\theta_k\in\Hom(C,X)$ for $k\le j$ satisfy the conditions (1)--(4).
  
  Now let $\theta=\sum_{j\in\bbN}\theta_j\in\Hom(C,X)$, which is
  well-defined since the sum for each component $C\to X_n$ is
  finite. For each $j\in\bbN$ we have
  $\p_{i_j}(\theta)=\p_{i_j}(\theta_j)+\p_{i_j}(\sum_{k>j}\theta_k)\neq
  0$, since the second summand lies in $(Y_{i_j})_{\g_{n_{j+1}}}$,
  whereas the first does not. On the other hand, the morphism
  $\p\theta$ factors through a finite sum $\coprod_{i\in J} Y_i$ for
  some $J\subseteq I$, since $C$ is finitely generated. This
  contradiction finishes the proof.
\end{proof}

We include the application from \cite{Ch1960} about products of
projective modules. Note that the descending chain condition on principal
right ideals characterises rings that are left perfect \cite{Ba1960}.

\begin{center}
  \includegraphics[width=.95\textwidth]{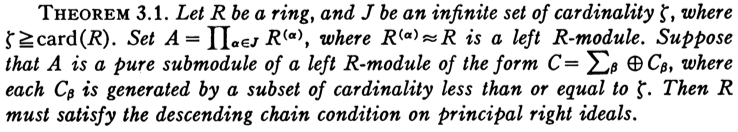}
\end{center}

\section{Coproducts of injective objects}

A motivation for Chase's study of products of projective modules in
\cite{Ch1960} was the fact that coproducts of injective modules are
again injective over a noetherian ring. In fact, this property for
right modules characterises right noetherian rings
\cite{Ma1958,Pa1959}. There are similar results for Grothendieck
categories, and this brings us back to Theorem~\ref{th:loc-noeth}. Roos
stated this theorem in \cite{Ro1968}, but again the proof is short:
\emph{La d\'emonstration du th\' eor\`eme~1 est analogue \`a celle du
  th\'eor\`eme~B de \cite{FW1969}}. For this reason it seems
appropriate to include a complete proof which is based on Chase's
lemma; it is different from that in \cite{FW1969}, though the authors
do refer to the work of Chase \cite{Ch1960}.

\begin{proof}[Proof of Theorem~\ref{th:loc-noeth}]
  Let $\A$ be a Grothendieck category and fix a generator $G$. When
  $\A$ is locally noetherian, then every injective object decomposes
  into a coproduct of indecomposable objects \cite[IV.2]{Ga1962}. Each
  indecomposable injective object arises as injective envelope
  $E(G/U)$ for some subobject $U\subseteq G$. The subobjects of any
  object in a Grothendieck form a set. Thus
  $E=\coprod_{U\subseteq G}E(G/U)$ has the property that every object
  of $\A$ is a subobject of a coproduct of copies of $E$, since $\A$
  admits injective envelopes.

  To prove the converse we need to assume that the Grothendieck
  category is \emph{locally finitely generated}, so it has a
  set of finitely generated generators.  Let $C\in\A$ be a finitely
  generated object.  We wish to show that $C$ is noetherian. To this
  end fix a chain of finitely generated subobjects
  $0=B_0\subseteq B_1\subseteq B_2\subseteq \cdots$ and set
  $C_n=C/B_n$. This yields a sequence of epimorphisms
  $\g=(C_{n}\twoheadrightarrow C_{n+1})_{n\in\bbN}$. For $X\in\A$ we
  set $X_{\bar\g_n}=\Hom(B_{n+1}/B_n,X)$ and obtain an exact sequence
\[0\lto X_{\g_{n+1}}\lto X_{\g_{n}}\lto X_{\bar\g_{n}}\lto 0\] provided that
$X$ is injective or a coproduct of injective objects.

Now consider a cogenerator $E$ such that each object of $\A$ embeds
into a coproduct of copies of $E$. We may assume that $E$ is injective
by replacing $E$ with its injective envelope. Let
$\kappa=\max(\aleph_0,\card\Hom(C,E))$ and choose a monomorphism
\[\p\colon \prod_{n\in\bbN }E^\kappa\lto \coprod_{i\in I}E.\] For each $m\in\bbN$ we apply
$\Hom(C_m,-)$ and obtain a monomorphism
\[\p_m\colon\prod_{n\in\bbN }(E_{\g_m})^\kappa\lto \coprod_{i\in
   I}E_{\g_m}\] since $X\mapsto X_{\g_m}$ preserves products and
 coproducts.  Then it follows from
 Lemma~\ref{le:chase} that for some $m\in\bbN$ the
 map $\p_m$ restricts to an embedding
\[\prod_{n\ge m }(E_{\g_m})^\kappa\lto (\coprod_{i\in J}E_{\g_\infty})\amalg(\coprod_{\text{finite}}E_{\g_m})\]
for some cofinite subset $J\subseteq I$, where
$E_{\g_\infty}=\bigcap_{n\geq 0} E_{\g_n}$. Comparing this with
$\p_{m+1}$ and passing to the quotient yields a commutative diagram
with exact rows
\[\begin{tikzcd}[column sep=small]
    0\arrow{r}&\prod\limits_{n\ge m }(E_{\g_{m+1}})^\kappa\arrow{r}\arrow{d}& \prod\limits_{n\ge m }(E_{\g_m})^\kappa\arrow{r}\arrow{d}&\prod\limits_{n\ge m }(E_{\bar\g_m})^\kappa\arrow{r}\arrow{d}&0\\
     0\arrow{r}& (\coprod\limits_{i\in  J}E_{\g_\infty})\amalg(\coprod\limits_{\text{finite}}E_{\g_{m+1}})\arrow{r}&
     (\coprod\limits_{i\in J}E_{\g_\infty})\amalg(\coprod\limits_{\text{finite}}E_{\g_m})  \arrow{r}&\coprod\limits_{\text{finite}}E_{\bar\g_m}\arrow{r}&0
   \end{tikzcd}\] where we use the fact that $E$ is injective.  The
 vertical map on the right is a monomorphism because it is a
 restriction of $\Hom(B_{m+1}/B_m,\p)$.  From the choice of $\kappa$
 it follows that $E_{\bar\g_m}=0$, cf.\ Lemma~\ref{le:groups-card}
 below.  Thus $C_m=C_{m+1}$ since $E$ cogenerates $\A$. We conclude
 that $C$ is noetherian.
\end{proof} 

\begin{lem}\label{le:groups-card}
Let $A$ be an abelian group with $\a=\card A$ and let
$\k\ge\max(\aleph_0,\a)$.
If there is a monomorphism $A^\k\to A^n$ for some $n\in\bbN$, then
$A=0$.
\end{lem}
\begin{proof}
Suppose $A\neq 0$.  Then we have
\[\card(A^\k)=\a^\k\ge 2^\k>\k=\k^n\ge \a^n=\card(A^n).\]
This  contradicts the fact that there is an injective map  $A^\k\to
A^n$.
\end{proof}

\begin{rem}
  The paper of Roos \cite{Ro1968} formulates
  Theorem~\ref{th:loc-noeth} for Grothendieck categories satisfying
  Grothendieck's condition (AB6).
\end{rem}

We end this note with some further references. Huisgen-Zimmermann
provides in \cite{HZ2000} a detailed survey about pure-injective
modules, emphasising the role of Chase's lemma. For a more recent
treatment of Chase's lemma and its generalisations we refer to work of
Bergman \cite{Be2015}.

\subsubsection*{Acknowledgement} This note is a
contribution to the series `Selected Topics in Representation Theory'
at Bielefeld University. I am grateful to Dieter Vossieck for
carefully reading a preliminary version and for several
helpful comments.

\end{document}